\documentclass[reqno]{amsart}

\usepackage{amsmath,amssymb,color,amsthm}
\usepackage{enumerate,enumitem}
\usepackage{booktabs,tabularx,multirow}
\usepackage{graphicx}
\usepackage{verbatim}
\usepackage{cite}
\usepackage[margin=1.5in]{geometry}

\newtheorem{theorem}{Theorem}

\newtheorem{lemma}[theorem]{Lemma}

\newtheorem{corollary}[theorem]{Corollary}

\newtheorem{proposition}[theorem]{Proposition}

\newtheorem*{claim*}{Claim}

\newtheorem{observation}[theorem]{Observation}

\newtheorem{prob}{Problem}

\newtheoremstyle{definition}% name
 {4pt}%   Space above
 {4pt}%   Space below
 {\sl}%   Body font
 {}%     Indent amount (empty = no indent, \parindent = para indent)
 {\bfseries}% Thm head font
 {.}%    Punctuation after thm head
 {.5em}%   Space after thm head: " " = normal interword space;
 {}%     Thm head spec (can be left empty, meaning `normal')

\theoremstyle{definition}
\newtheorem{definition}[theorem]{Definition}

\theoremstyle{remark}

\newtheorem*{remark*}{Remark}

\newtheoremstyle{introthms}% name
 {3pt}%   Space above, empty = `usual value'
 {3pt}%   Space below
 {\itshape}% Body font
 {}%     Indent amount (empty = no indent, \parindent = para indent)
 {\bfseries}% Thm head font
 {.}%    Punctuation after thm head
 {.5em}%   Space after thm head: " " = normal interword space;
 {\thmnote{#3}}% Thm head spec
\theoremstyle{introthms}
\begin{document}
\title{Online and size anti-Ramsey numbers}

\author[M. Axenovich]{Maria Axenovich}
\address{Karlsruher Institut f\"ur Technologie, Karlsruhe, Germany}
\email{maria.aksenovich@kit.edu}

\author[K. Knauer]{Kolja Knauer}\thanks{The second author was supported by the ANR TEOMATRO grant ANR-10-BLAN 0207}
\address{Universit\'e Montpellier 2, Montpellier, France}
\email{kolja.knauer@math.univ-montp2.fr}

\author[J. Stumpp]{Judith Stumpp}
\address{Karlsruher Institut f\"ur Technologie, Karlsruhe, Germany}

\author[T. Ueckerdt]{Torsten Ueckerdt}
\address{Karlsruher Institut f\"ur Technologie, Karlsruhe, Germany}
\email{torsten.ueckerdt@kit.edu}

\date{\today}
 
\begin{abstract}
 A graph is properly edge-colored if no two adjacent edges have the same color. The smallest number of edges in a graph any of whose proper edge colorings contains a totally multicolored copy of a graph $H$ is the {\bf size anti-Ramsey number} $AR_s(H)$ of $H$. This number in offline and online setting is investigated here.
\end{abstract}

\maketitle


\section{Introduction}

In the coloring problems of this paper we assume that the edges of graphs are colored with  natural numbers. A graph is \emph{rainbow} or \emph{totally multicolored} if all its edges have distinct colors. For a graph $H$, let $AR(n,H)$, the {\bf classical anti-Ramsey number}, be the largest number of colors used on the edges of a complete $n$-vertex graph without containing a rainbow copy of $H$. This function was introduced by Erd\H{o}s, Simonovits and S\'os~\cite{ESS} and it is shown to be finite. In particular, having sufficiently many colors on the edges of a complete graph forces the existence of a rainbow copy of $H$ and this number of colors is closely related to the Tur\'an or extremal number for $H$. 
% 
%  The developed anti-Ramsey theory concludes that using sufficiently many colors on the edges of a complete graph forces the existence of a rainbow copy of $H$ and this number of colors is closely related to the Tur\'an or extremal number for $H$. 

An important and natural class of colorings consists of {\it proper colorings}. A proper coloring of a graph assigns distinct colors to adjacent edges. Here, we are concerned with the most efficient way to {\bf force a rainbow copy} of $H$ in properly colored graphs. It is clear that to force a rainbow copy of a graph one does not necessarily need to use a complete graph to be colored. For that, we investigate the smallest number of edges in a graph such that any proper coloring contains a rainbow copy of $H$. We introduce a size and online version of anti-Ramsey numbers for graphs, denoted $AR_s$ and $AR_o$, respectively.  
% 
% In addition, we study this question in an online setting.
% 
Note that these very natural variants were studied for Ramsey numbers, see~\cite{C,EFRS,KK,EF92}, to name a few. There are interesting connections between classical, size and online Ramsey numbers. However, in an anti-Ramsey setting, these functions were not considered so far. This paper initiates this study.

\smallskip

Formally, we define the {\bf size anti-Ramsey number}, denoted by $AR_s(H)$, to be the smallest number of edges in a graph $G$ such that any of its proper edge-colorings contains a rainbow copy of $H$.
The function $AR_o(H)$, the {\bf online anti-Ramsey number}, is defined in the following game setting. There are two players, Builder and Painter. Builder creates vertices and edges, one edge at a time. After each new edge has been built by Builder, Painter assigns a color to it such that no two adjacent edges have the same color. Both, Builder and Painter have complete information, that is, at all times the so-far built graph and its coloring is known to both players. Builder wins after $t$ steps if after the $t$-th edge has been built and colored there is a rainbow copy of a graph $H$. We shall show below that Builder can always win and the smallest such $t$ is denoted by $AR_o(H)$.

In studying size and online anti-Ramsey numbers, we make use of two additional functions: $AR_{FF}$ and $AR_{loc}$. Here, $AR_{FF}(H)$ is the {\bf First-Fit anti-Ramsey number} and it is defined as $AR_o(H)$ with the exception that Painter must follow the First-Fit strategy, i.e., Painter must use the smallest feasible color. In addition, $AR_{loc}(H)$ is the {\bf local anti-Ramsey number}, the smallest $n$ such that any proper edge coloring of $K_n$ contains a rainbow copy of $H$, see for example Babai~\cite{B}, Alon, Lefmann, R{\"o}dl~\cite{ALR}. In all the above anti-Ramsey numbers we consider only simple graphs unless stated differently. 

We observe an easy chain of inequalities for graphs $H$ without isolated vertices, which we prove formally in Theorem~\ref{thm:inequalities}:
\begin{equation}\label{eq:inequalities}
 \frac{AR_{loc}(H)^{1/3}}{2} \leq AR_{FF}(H)\leq  AR_o(H) \leq  AR_s(H) \leq  \binom{AR_{loc}(H)}{2}.
\end{equation}
% 

% \smallskip

\noindent
\textbf{Our Results.}
We provide bounds on $AR_s(H), AR_o(H)$ and $AR_{FF}(H)$ for general graphs $H$, as well as for some specific graphs including the complete graph $K_n$ on $n$ vertices, the path $P_k$ on $k$ edges, the cycle  $C_k$ on $k$ edges, and the matching $M_k$ on $k$ edges. For the terminology used in this paper we refer to the standard textbook of West~\cite{W}.

We summarize our main results in Table~\ref{tab:overview} and provide the corresponding theorems in respective sections.

%%%%%%%%%%%%%%%%%%%%%%%%%%%%%%%%%%%%%%%%%%%   
%%%%%%%%%                       %%%%%%%%%%%
%%%%%%%%%   The TABLE           %%%%%%%%%%%
%%%%%%%%%                       %%%%%%%%%%%
%%%%%%%%%%%%%%%%%%%%%%%%%%%%%%%%%%%%%%%%%%% 
\begin{table}[h!]
 \renewcommand{\arraystretch}{1.25}
%  \vspace{1em}
 \newcolumntype{C}{>{\centering\arraybackslash}X}%
 \begin{tabularx}{\textwidth}{cCCC}
  \toprule
  $H$ & $AR_{FF}(H)$ & $AR_o(H)$ & $AR_s(H)$ \\
  \midrule
  \multirow{3}{*}{\begin{minipage}{6em}
          \centering
          $n$-vertex\\ $k$-edge graph
         \end{minipage}} 
  & \multirow{3}{*}{\begin{minipage}{6em}
          \centering
          $\leq (\tau+1)\cdot k$ \\[5pt] $\geq \displaystyle\frac{k^2}{4(\Delta+1)}$
         \end{minipage}} 
  & \multirow{3}{*}{$\leq k^2$} 
  & $\leq 8k^4/n^2$
  \\
  & & & $\leq \frac{k(k+1)}{2}^*$
  \\[2pt]
  & & & $\geq \frac{(k-1)^2}{2\Delta}$
  \\
  \midrule
  \multirow{3}{*}{$P_k$} 
  & \multirow{3}{*}{$\displaystyle\frac{k^2}{8}(1 + o(1))$} 
  & \multirow{3}{*}{$\displaystyle\leq \frac{k^2}{2}$} 
  & {$\leq \frac{8}{9}k^2$~\cite{GM}}
  \\
  & & & {$\leq \frac{k^2}{2}^*$} 
  \\[2pt]
  & & & {$\geq \frac{k(k+1)}{4}$} 
  \\
  \midrule
  \multirow{3}{*}{$C_k$} & \multirow{3}{*}{$\displaystyle\frac{k^2}{8}(1 + o(1))$} 
  & \multirow{3}{*}{$\leq \displaystyle\frac{k^2}{2}$}
  & {$\leq k^2$}
  \\
  & & & {$\leq \frac{k^2}{2}^*$} 
  \\[2pt]
  & & & {$\geq \frac{k(k-1)}{4}$} 
  \\
  \midrule
  \multirow{2}{*}{$M_k$} 
  & \multirow{2}{*}{$\displaystyle\left\lceil\frac{k^2}{4}\right\rceil$}
  & \multirow{2}{*}{$\displaystyle\leq \frac{k^2}{3}(1 + o(1))^*$} 
  & \multirow{2}{*}{$\displaystyle\frac{k(k+1)}{2}$} 
  \\
  & & & 
  \\
  \midrule
  \multirow{2}{*}{$K_n$}
  & \multirow{2}{*}{$\displaystyle\frac{n^3}{6}(1 + o(1))$} 
  & \multirow{2}{*}{$\displaystyle\leq \frac{n^4}{4}$} 
  & $\leq c n^6/ \log^2 n$~\cite{AJMP}
  \\
  & & & $\geq c n^5/ \log n$
  \\
  \bottomrule
 \end{tabularx}
 \vspace{0.5em}
 \caption{Overview over online, First-Fit and size anti-Ramsey numbers for paths, cycles, matchings, and complete graphs. Upper bounds marked with $^*$ are derived from graphs $G$ with multiple edges that enforce rainbow $H$. Here $c$ denotes an absolute positive constant, $\tau$ the vertex cover number and $\Delta$ the maximum degree of $H$.}
 \label{tab:overview}
\end{table}

\noindent
\textbf{Related Work.}
% 
% To conclude the introductory section, let us mention some other related results.
% 
The anti-Ramsey numbers and their variations were studied in a number of papers, see~\cite{FMO} for a survey. The problem of finding rainbow subgraphs in properly-colored graphs was originally motivated by the transversals of Latin squares,~\cite{R}. The largest number of edges in a graph on $n$ vertices that has a proper coloring without rainbow $H$, referred as rainbow Tur\'an number, was studied by Das, Lee and Sudakov~\cite{DLS} and Keevash, Mubayi, Sudakov and Verstra{\"e}te~\cite{KMSV}. Clearly, we have $AR_s(H) \leq {\rm ex *}(H) +1$ since if a graph $G$ has ${\rm ex *}(H) +1$ edges, then any (!) of its proper edge-colorings will contain a rainbow copy of $H$. However, $AR_s(H)$ can be arbitrarily smaller than $\rm{ex*}(n, H)$ when $n$ is not fixed. To see this, consider $H=M_2$, a matching with two edges. Any star does not contain $M_2$, so $\rm{ex*}(n, H)\geq n-1$, however, $AR_s(H)=3$ as we shall see in Section~\ref{sec:size}. So, these two functions are completely different in nature.

Finding rainbow copies of a given graph in non-necessarily properly colored graph was investigated as well. 
The conditions on the colorings that force rainbow copies include large total number of colors, 
large number of colors incident to each vertex,  small number of edges in each color class,  bounded number of edges in each monochromatic star, see~\cite{ESS,AJT,AJMP,BKP}, to name a few. Among some perhaps  unexpected results on anti-Ramsey type numbers
is the fact that  rainbow cycles can be forced in properly colored graphs of arbitrarily high girth, see~\cite{RT} and 
that rainbow subgraphs in edge-colored graphs are related  to matroid theory, see~\cite{BL,T}. 

\smallskip

In this work, we are concerned with size and online anti-Ramsey numbers, not studied so far.  In comparison, 
the online Ramsey number has been introduced by Beck~\cite{Bc} in 1993 and independently by Friedgut \textit{et al.}~\cite{FKRRT} in 2003 and Kurek and Ruci{\'n}ski~\cite{KR} in 2005. Here Builder presents some graph one edge at a time, and Painter colors each edge as it appears either red or blue. The goal of Builder is to force a monochromatic copy of a fixed target graph $H$ and the goal of Painter is to avoid this as long as possible. The online Ramsey number of $H$ is then the minimum number of edges Builder has to present to force the desired monochromatic copy of $H$, no matter how Painter chooses to color the edges, see~\cite{C,GKP,P,GHK}.
 The size-Ramsey number  for a graph $H$ is the smallest number of edges in a graph $G$ such that any $2$-coloring of the edges of $G$ results in a monochromatic copy of $H$, 
see~\cite{EFRS, B83}.    One of the fundamental questions studied in size-Ramsey and online Ramsey  theory is how these numbers relate to each other and to the 
classical Ramsey numbers and how they are expressed in terms of certain graph parameters.

Here, we initiate the study of size- and online anti-Ramsey numbers answering similar questions:
\textit{``How do the various anti-Ramsey numbers relate to each other?''},   \textit{``What are the specific values for classes of graphs?''}, 
\textit{``What graph parameters play a determining role in studying these anti-Ramsey type functions?''}.

\smallskip

\noindent
\textbf{Organization of the Paper.}
 
\begin{itemize}

 \item{}In Section~\ref{sec:comparison} we prove the inequalities~\eqref{eq:inequalities}. We also investigate the relation between the local anti-Ramsey number $AR_{loc}$ and the size anti-Ramsey number $AR_s$. By recalling known upper bounds for $AR_{loc}$ we obtain some upper bounds for $AR_s$ using~\eqref{eq:inequalities}.
 
 \item{}In Section~\ref{sec:size} we only consider the size anti-Ramsey numbers. We prove upper and lower bounds on $AR_s(K_n)$, provide general bounds on $AR_s(H)$ in terms of the chromatic index $\chi'(H)$ and the vertex cover number $\tau(H)$. We give upper and lower bounds for $AR_s(P_k)$ and $AR_s(C_k)$ and determine $AR_s(M_k)$ exactly for every $k$. Moreover, we show how rainbow copies can be forced using fewer edges if multiple edges are allowed.
 
 \item{}In Section~\ref{sec:first-fit} we consider the First-Fit anti-Ramsey number $AR_{FF}$. We show that it is determined up to a factor of $4$ by a certain matching parameter of $H$ and prove asymptotically tight bounds for paths, cycles, matchings and complete graphs.
 
 \item{}In Section~\ref{sec:online}  we consider the online anti-Ramsey number $AR_o$. We present a strategy for Builder that gives a general upper bound on $AR_o(H)$ in terms of $|E(H)|$. We also show how to improve this bound in case when $H$ is a matching.
  
\end{itemize}

We drop floors and ceilings when appropriate   not  to congest the presentation. We also assume without loss of generality that in all online settings the graph $H$ that is to be forced has no isolated vertices.

%%%%%%%%%%%%%%%%%%%%%%%%%%%%%%%%%%%%%%%%%  
%%%                                   %%%
%%% COMPARISON OF ANTI-RAMSEY NUMBERS %%%
%%%                                   %%%
%%%%%%%%%%%%%%%%%%%%%%%%%%%%%%%%%%%%%%%%%
\section{Comparison of Anti-Ramsey Numbers}\label{sec:comparison}

In this section we investigate the relation between the four concepts of anti-Ramsey numbers: First-Fit anti-Ramsey number $AR_{FF}$, online anti-Ramsey number $AR_o$, size anti-Ramsey
number $AR_s$ and local anti-Ramsey number $AR_{loc}$. We present several immediate inequalities that are valid for all (non-degenerate) graphs $H$, and conclude upper bounds on the new anti-Ramsey numbers for complete graphs, general graphs, paths, cycles and matchings.
We shall write $G\rightarrow H$ if every proper edge-coloring of $G$ contains a rainbow copy of $H$.

\begin{theorem}\label{thm:inequalities}
 For any graph $H$ without isolated vertices we have
 \begin{itemize}
  \itemsep6pt
  \item $\frac{AR_{loc}(H)^{1/3}}{2} \leq AR_{FF}(H)\leq AR_o(H) \leq AR_s(H) \leq \binom{AR_{loc}(H)}{2}$
  \item $AR_s(H) \geq \frac{1}{2} AR_{loc}(H) \delta(H),$
 \end{itemize}
 where $\delta(H)$ denotes the minimum degree of $H$.
\end{theorem}
\begin{proof}
For the first inequality note that $\delta(H) \geq 1$ implies $|E(H)| \geq |V(H)|/2$. Then clearly $AR_{FF}(H) \geq |E(H)| \geq |V(H)|/2$. On the other hand, Alon~\textit{et al.}~\cite{AJMP} proved that $AR_{loc}(H) \leq |V(H)|^3$, which proves the claimed inequality. We remark that taking the better upper bound $AR_{loc}(K_n) \leq cn^3/\log(n)$ from~\cite{AJMP} we get the slightly better bound $AR_{FF}(H) \geq C(AR_{loc}(H) \log AR_{loc}(H))^{1/3}$ for some constant $C > 0$.

 To see the second inequality observe that Builder has an advantage if she knows the strategy of Painter. More precisely, if for \emph{every} strategy of Painter Builder can create a graph $G$ and an order of the edges of $G$ that forces a rainbow $H$, in particular she can do so for First-Fit. Hence $AR_{FF}(H) \leq AR_o(H)$.

For the third inequality consider $G$ on $AR_s(H)$ edges such that $G \rightarrow H$. Clearly, if Builder exposes the edges of $G$ in any order, then every coloring strategy of Painter will produce a proper coloring of $G$ and thus a rainbow $H$. In other words, $AR_o(H) \leq AR_s(H)$.

For the last inequality, recall that $AR_s(H)$ is the minimum number of edges in any graph $G$ with $G \rightarrow H$. Since  $G$ could always be taken as a complete graph,   $AR_s(H) \leq \binom{AR_{loc}(H)}{2}$.

 \smallskip
 
 For the inequality in the second item let $G\rightarrow H$ and $G$ have the smallest number of edges with that property. Then $\delta(G) \geq \delta(H)$. Otherwise, if $v$ was a vertex of $G$ with degree less than $\delta(H)$, coloring $G-v$ properly without a rainbow copy of $H$ and coloring the edges incident to $v$ arbitrarily and properly gives a coloring of $G$ without a rainbow copy of $H$. Observe also that $|V(G)|\ge AR_{loc}(H)$ because otherwise there is a complete graph on $AR_{loc}(H)-1$ vertices all of its proper colorings contain a rainbow $H$. Together with the lower bound on the degree this yields the claim.
\end{proof}

\begin{remark*}%\label{rem:isolated}
All of the inequalities in Theorem~\ref{thm:inequalities} except for the very first one are also valid if $H$ contains isolated vertices. The first inequality is not valid in this case because $AR_{FF}(H)$ measures the number of edges in a graph $G$ containing rainbow $H$ and $AR_{loc}(H)$ measures the number of vertices in $G$ containing rainbow $H$. So when $H = I_n$ is an $n$-element independent set, then clearly $AR_{FF}(I_n) = 0$ while $AR_{loc}(I_n) = n$. 
\end{remark*}

The following proposition shows that $AR_s(H)$ can differ significantly from $\binom{AR_{loc}(H)}{2}$.

\begin{proposition}\label{prop:n=4}
 $$AR_s(K_4) = 15 \quad \text{and} \quad \binom{AR_{loc}(K_4)}{2}=\binom{7}{2} = 21.$$
\end{proposition}
\begin{proof}
 To see that $AR_{loc}(K_4)>6$, observe that $K_6$ can be properly colored with $5$ colors, but a rainbow $K_4$ requires $6$ colors. 
 Let $G$ be a graph on $7$ vertices, with vertex set $X\cup Y$,  $X$ induces a triangle, $Y$ induces an independent set on $4$ vertices and 
 $(X,Y)$ is a complete bipartite graph.  
Consider an arbitrary  proper coloring of $G$ and assume that $X$  induces color $1, 2, 3$. 
Since there are at most three edges of colors from $\{1, 2, 3\}$ between $X$ and $Y$, there is a vertex  $v\in Y$  sending edges of colors different from $1,2$, and $3$ to $x,y,z$.
 Thus $\{v, x, y, z\}$ induces a rainbow $K_4$. 
 So, we have that $G\rightarrow K_4$ and thus $AR_{loc}= 7$ and  $AR_s(K_4) \leq 15$.

 \smallskip

 To see that any graph with at most $14$ edges can be properly colored without a rainbow $K_4$, we use induction on the number of vertices and edges of $G$ with the basis 
 $|V(G)|=6$ shown above.  Now,  we assume that  $G$ has at least $7$ vertices and any of its subgraphs can be properly colored without rainbow $K_4$, however
 $G$ itself can not be properly colored avoiding rainbow $K_4$.
 It is easy to see that  $G$ has minimum degree at least $3$ and each edge belongs to at least $2$ copies of $K_4$. 
 Consider the copies of $K_4$ in $G$.  Assume first that there are two such copies sharing three vertices. Then together they span $9$ edges and 
 $5$ vertices. Thus, the remaining set of at least $2$ vertices is incident to at most $5$ edges.  Since these vertices are of degree at least $3$, there must be exactly $2$ of them, 
 they both must be of degree exactly $3$, and they must be adjacent. Thus, they belong to at most one copy of $K_4$, a contradiction.
 Thus, any two copies of $K_4$ share at most two vertices.  
 Let $K$ be the graph whose vertices are the copies of $K_4$ in $G$ and two vertices of $K$  are adjacent if and only if the corresponding copies share an edge.
 There is a component of $K$ with at least three vertices, otherwise we can easily color $G$ avoiding rainbow $K_4$.
 This component corresponds to at  least $6+ 5 + 4=15$ edges in $G$, a contradiction.
  \end{proof}

%%%%%%%%%%%%%%%%%%%%%%%%%%%%%%%%%%%%%%%%%%%  
%%%%%%%              %%%%%%%%
%%%%%%% SIZE ANTI-RAMSEY NUMBERS %%%%%%%%
%%%%%%%              %%%%%%%%
%%%%%%%%%%%%%%%%%%%%%%%%%%%%%%%%%%%%%%%%%%%
\section{Size Anti-Ramsey Numbers}\label{sec:size}

The inequalities in Theorem~\ref{thm:inequalities} together with known upper bounds on the local anti-Ramsey number $AR_{loc}$ immediately give upper bounds on $AR_s$, $AR_o$ and $AR_{FF}$. We shall first collect some known bounds on $AR_{loc}(H)$ for several graphs $H$, and then conclude corresponding upper bounds on $AR_s(H)$ in Corollary~\ref{cor:from-loc} below.

\smallskip

Alon~\textit{et al.}~\cite{AJMP} prove that there are absolute constants $c_1, c_2 >0$, such that
\begin{equation}
 c_1 n^3/\log n \leq AR_{loc}(K_n) \leq c_2 n^3/\log n.\label{eq:local-complete}
\end{equation}
See also earlier papers of Alon, Lefmann, R{\"o}dl~\cite{ALR} and Babai~\cite{B}. In the same paper,~\cite{AJMP}, the authors prove that for any graph $H$ on $n$ vertices and $k \geq 2$ edges
$$k-1 \leq AR_{loc}(H) \quad \text{and} \quad \frac{ck^2}{n \log n} \leq AR_{loc}(H) \leq 4k^2/n \quad \text{for some $c>0$}.$$
Here, the upper bound is obtained by a greedy counting and the lower bound is given by a probabilistic argument.

The local anti-Ramsey number for the path $P_k$ on $k$ edges was investigated by Gy\'arf\'as \textit{et al.}~\cite{GRSS} and was conjectured to be equal to $k+c$ for a small constant $c$. Indeed it seems to be widely believed that it is $k+2$ (see e.g.~\cite{GM}), which would be best-possible since Maamoun and Meyniel~\cite{MM} show that for $n$ being a power of $2$ there are proper colorings of $K_n$ admitting no Hamiltonian rainbow path. Frank Mousset gave the best currently known upper bound in his bachelor's thesis~\cite{GM}, so that together we have
$k+2 \leq AR_{loc}(P_k) \leq \frac{4}{3} k + o(k).$
Gy\'arf\'as \textit{et al.}~\cite{GRSS} considered the cycle $C_k$ on $k$ edges and showed that 
$k+1 \leq AR_{loc}(C_k) \leq \frac{7}{4} k +o(k).$
For a matching $M_k$ on $k$ edges, Kostochka and Yancey~\cite{KY}, see also Li, Xu~\cite{LX1}, 
proved that $AR_{loc}(M_k) = 2k.$

From these bounds and the inequality $AR_s(H) \leq \binom{AR_{loc}(H)}{2}$ for all graphs $H$ from Theorem~\ref{thm:inequalities} we immediately get the following.

\begin{corollary}\label{cor:from-loc}
For an $n$-vertex $k$-edge graph $H$, $AR_s(H) \leq 8k^4/n^2$. 
Moreover,  $AR_s(K_n) \leq cn^6/\log^2 n$,  ~$~~AR_s(P_k) \leq \binom{4k/3}{2} + o(k^2)$,  $ ~~AR_s(C_k) \leq \binom{7k/4}{2} + o(k^2)$, and 
$ ~~AR_s(M_k) \leq 2k^2.$
\end{corollary}

Next in this section we improve some of these bounds.

\begin{theorem}
 There exists positive constants $c_1,c_2$ such that
 $$c_1n^5/ \log n \leq AR_s(K_n)\leq c_2 n^6/ \log^2 n.$$
\end{theorem}
\begin{proof}
 The upper bound follows from  Corollary~\ref{cor:from-loc}.  For the lower bound,  we  shall show that  there is a positive constant $c$ such that 
 any graph $G$ on less than $cn^5/\log n$ edges can be  properly  colored  avoiding   a rainbow copy of $K_n$. 
 %We choose $c = \frac{c_1}{16\cdot 8}$ with the constant $c_1$ as in~\eqref{eq:local-complete}.
 Let $\ell = \binom{n/2}{2}$,  $V_1$ be the set of vertices of $G$ of degree at most $\ell-2$ and  $V_2=V(G) \setminus V_1$. 
 We shall properly color the edges induced by $V_1$, $V_2$ and the edges between $V_1$ and $V_2$ separately.
Color $G[V_1]$ properly  with at most $\ell-1$ colors using Vizing's theorem. Since $\ell -1 < \binom{n/2}{2}$, there is no rainbow $K_{n/2}$ in $G[V_1]$.
 Because $|E(G)| < cn^5/\log n$ and every vertex in $V_2$ has degree at least $\ell-1$ we have
 $|V_2| \leq \frac{2|E(G)|}{\ell-1} < c' \frac{n^3}{\log n} \leq  c''  \cdot \frac{(n/2)^3}{\log(n/2)}.$
  Hence the lower bound in~\eqref{eq:local-complete} guarantees that there is a proper coloring of $G[V_2]$ without rainbow copy of $K_{n/2}$ for an 
  appropriate constant.  Now, coloring the edges between $V_1$ and $V_2$ arbitrarily so that the resulting coloring is proper, we see that $G$ has no rainbow copy of $K_n$ under the described coloring because this copy would have at least $n/2$ vertices either in $V_1$ or $V_2$. However, neither $G[V_1]$ nor $G[V_2]$ contains a rainbow copy of $K_{n/2}$.
\end{proof}

To get general lower bounds on $AR_s(H)$, we can use lower bounds on $AR_{FF}(H)$, which we present in Section~\ref{sec:first-fit} below. For some sparse graphs $H$, such as matchings and paths, this   gives the correct order of magnitude. Note that $AR_{FF}(K_{1,k}) = AR_s(K_{1,k})=k$, so there is no meaningful lower bound just in terms of the number of edges of $H$. However, the maximum degree could help.

\begin{lemma}\label{lem:lowersize}
 Let $H$ be a graph on $k$ edges with maximum degree $\Delta$. Then
 $AR_s(H) \geq \frac{(k-1)^2}{2\Delta}.$
 Moreover, if $H$ has chromatic index $\chi'(H) = \Delta$, then
 $AR_s(H) \geq \frac{k(k+1)}{2\Delta}.$
\end{lemma}
\begin{proof}
 Let $G$ be a graph with $G\rightarrow H$. We shall find several copies of $H$ in $G$ one after another and argue that as long as we do not reach a certain number of copies we can color  their edges properly with less than $k$ colors, i.e., without a rainbow copy of $H$.
 Without loss of generality we assume that $\chi'(H) = \chi' < k$, which means that $H$ is neither a star nor a triangle. Note that for graphs with $\chi' = k$ we have $AR_s(H) = k$ and the claimed inequalities hold.
 
Let $H_0$ be any copy of $H$ in $G$. Edge-color $H_0$ properly with $\chi'$ colors. Because $\chi' < k$,  $H_0$ is not rainbow. 
Extending this coloring arbitrarily to a proper coloring of $G$ gives some rainbow copy $H_1$ of $H$ since $G\to H$. 
Since $H_1$ is rainbow it shares at most $\chi'$ colors with $H_0$. Thus $H_1\setminus H_0$ has at least $k-\chi'$ edges. 
Color $H_0 \cup H_1$ properly with $\chi'(H_0 \cup H_1)$ colors.  Generally, for $i>1$ after having colored $\hat{H_i} = \cup_{0 \leq j < i}H_j$ with $\chi'(\hat{H_i})$ colors
 and extending this to a proper coloring  of $G$, there is a rainbow copy $H_i$ of $H$ in $G$ sharing at most $\chi'(\hat{H_i})$ edges with $\hat{H_i}$ and thus contributing 
 at least $k-\chi'(\hat{H_i})$ new edges to $G$. Since for each $i\geq 1$ the graph $\hat{H_i}$ is the union of $i$ copies of $H$ we have 
 $\Delta(\hat{H_i}) \leq i\Delta$ and hence $\chi'(\hat{H_i}) \leq i\Delta + 1$.\\
Let $\ell$ be the  smallest integer  for which $\chi'(\hat{H_\ell}) \geq k$.  With $\chi'(\hat{H_\ell}) \leq \ell \Delta +1$ we obtain $\ell = \lceil\frac{k-1}{\Delta}\rceil$ and conclude that
 \begin{align*}
 AR_s(H) 
  &\geq |E(G)| 
   \geq |E(\hat{H_\ell})|
   \geq |E(H_0)| + \sum_{i=1}^{\ell-1}(k - \chi'(\hat{H}_i)) =  \ell k - \sum_{i=1}^{\ell-1}     \chi'(\hat{H}_i)\\
  &\geq \ell k - \sum_{i=1}^{\ell-1} (i \Delta +1)
   = \ell k - \Delta \binom{\ell}{2} - (\ell-1)
   > \ell(k - \frac{\Delta}{2}(\ell-1) - 1) 
  %&= \left\lceil \frac{k-1}{\Delta} \right\rceil \left( k - \frac{\Delta}{2} \left(\left\lceil \frac{k-1}{\Delta} \right\rceil -1 \right) - 1\right) \\
 % &\geq \frac{k-1}{\Delta}\left( k - \frac{k-1}{2} - 1 \right) 
   \geq \frac{(k-1)^2}{2\Delta}.
 \end{align*}
 % Since $\hat{H_\ell}$ is the union of $\ell$ copies of $H$ we have $\Delta(\hat{H_\ell}) \leq \ell\Delta$ and hence $\chi'(\hat{H_\ell}) \leq \ell\Delta + 1$, which gives $\ell = \lceil\frac{k-1}{\Delta}\rceil$. Plugging this into~\eqref{eq:lowersize} we obtain
% 
 This proves the first part of the statement. Now if $\chi'(H) = \Delta$, then $\chi'(\hat{H_i}) \leq i\chi'(H) = i\Delta$ and thus $\ell = \lceil \frac{k}{\Delta} \rceil$, which gives in a
 similar calculation that $AR_s(H)\geq {k(k+1)}/{2\Delta}.$
%% 
% \begin{align*}
% AR_s(H)
%  &\geq |E(H_0)| + \sum_{i=1}^{\ell-1}(k - \chi'(H_i)) %\\
%   \geq \ell k - \sum_{i=1}^{\ell-1} i \Delta 
%%  = \ell k - \Delta \binom{\ell}{2}
%   = \ell(k - \frac{\Delta}{2}(\ell-1)) \\
%  &= \left\lceil \frac{k}{\Delta} \right\rceil \left( k - \frac{\Delta}{2} \left(\left\lceil \frac{k}{\Delta} \right\rceil -1 \right) \right) %\\
%  \geq \frac{k}{\Delta}\left( k - \frac{k-1}{2} \right)
%  = \frac{k(k+1)}{2\Delta}.
% \end{align*}
 \end{proof}

Next we derive several upper bounds on $AR_s(H)$ for some specific graphs $H$, namely paths, cycles and matchings. Recall from Corollary~\ref{cor:from-loc} that $AR_s(P_k) \leq \binom{4k/3}{2} + o(k^2) = \frac{8}{9}k^2 + o(k^2)$, which simply follows from the best known bound on $AR_{loc}(P_k)$. Apparently, this remains the best-known upper bound. But we can improve the bound $AR_s(C_k) \leq \binom{7k/4}{2} + o(k^2) = \frac{49}{32}k^2 + o(k^2)$ from Corollary~\ref{cor:from-loc}.

\begin{theorem}\label{thm:sizePathCycle}
 For every $k\geq 2$,    $\frac{k(k-1)}{4} \leq AR_s(P_{k-1}) \leq AR_s(C_k) \leq k^2$, and  $AR_s(M_k) =\binom{k+1}{2}$.
\end{theorem}
\begin{proof}
 The lower bound on $AR_s(P_{k-1})$ follows directly from Lemma~\ref{lem:lowersize}. And since $P_{k-1}$ is a subgraph of $C_k$ we have that $AR_s(P_{k-1}) \leq AR_s(C_k)$.

 For the upper bound $AR_s(C_k) \leq k^2$ we first consider the case that $k$ is even. 
 We shall construct $G$ such that $G\rightarrow C_k$.
 Let $V(G) = \{v_1, \ldots, v_{k/2}\} \cup W_1 \cup \cdots \cup W_{k/2}$, where 
 the union is disjoint and $|W_i|=4(i - 1)$,  for $i=2, \ldots, k/2$ and $|W_1|= 1$.
 Let the edge set of $G$ consist of the union of complete bipartite graphs  with parts $\{v_i, v_{i+1}\}, W_i$, for $i=1, \ldots, k/2-1$ and
 $\{v_1, v_{k/2}\},  W_{k/2}$. We refer to Figure~\ref{fig:path_size_UB} for an illustration of the graph $G$.
 
 \begin{figure}[htb]
  \centering
  \includegraphics{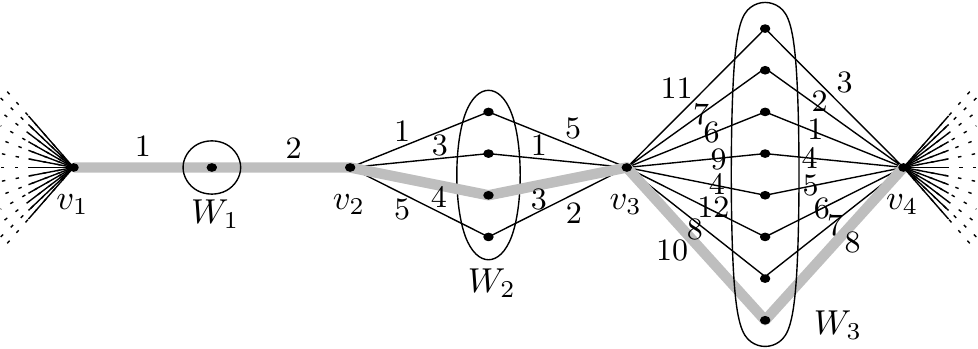}
  \caption{The portion of the graph $G$ with $G \to C_k$ consisting of the first three complete bipartite subgraphs and a proper coloring of it. The rainbow paths $R_1$, $R_2$ and $R_3$ are highlighted.}
  \label{fig:path_size_UB}
 \end{figure}

% Our graph $G$ with $G \to C_k$ contain the vertices $v_0,v_2,\ldots,v_{k-2},v_k = v_0$ and for $i = 1,\ldots,k/2$ between vertex $v_{2i-2}$ and $v_{2i}$ a number of $\max(1,4(i-1))$ paths of length $2$. Fix a proper coloring of $G$. We claim that for $i = 1,\ldots,k/2$ the subgraph $G_i$ of $G$ consisting of all the $2$-paths between vertices in $\{v_0, \ldots, v_{2i}\}$ contains a rainbow path on $2i$ edges with endpoints $v_0$ and $v_{2i}$, respectively. For $i = k/2$ this path is indeed a rainbow cycle on $k$ edges.
%
% Our claim holds for $i=1$ since $G_1$ is itself a properly colored $P_2$. For $i>1$ we consider any rainbow path $P$ in $G_{i-1}$ on $2(i-1)$ edges from $v_0$ to $v_{2(i-1)}$. Now there is at least one $2$-path between $v_{2(i-1)}$ and $v_{2i}$ that does not contain a color from $P$. This is because each of the $2(i-1)$ colors from $P$ appears at most once at $v_{2(i-1)}$ and at most once at $v_{2i}$, at least one such color at $v_{2(i-1)}$ is on a $2$-path to $v_{2(i-2)}$ and there are in total $4(i-1)$ $2$-paths between $v_{2(i-1)}$ and $v_{2i}$. Hence extending $P$ by the $2$-path between $v_{2(i-1)}$ and $v_{2i}$ that does not contain any color from $P$ we get a rainbow path on $2i$ edges as desired. It remains to count the edges in $G$:
 %
 
 We shall identify a rainbow $C_k$ greedily by picking a rainbow path $R_1= v_1, w_1, v_2$, $w_1\in W_1$, 
 then a rainbow path $R_2 = v_2, w_2, v_3$,\;\;$w_2 \in W_2$, such that the colors of $R_2$ and $R_1$ are disjoint and so on. 
 Then $|E(G)| = 2+ \sum_{i=2}^{k/2} 8(i-1)   \leq k^2 - 2k +2$.

 This concludes the proof when $k$ is even. Contracting one of the two edges in the only $2$-path between $v_0$ and $v_2$, we obtain a graph $G'$ with $G' \to C_{k-1}$, 
 and with at most $k^2 - 2k +1  = ((k-1)+1)^2 - 2(k-1) -2 +1 = (k-1)^2 $ edges.
 This completes the proof for  odd cycles.
 
 To show the upper bound  for matchings, consider the graph $G_k$ that is the vertex  disjoint union of $k$ stars on $1, 2, \ldots, k$ edges respectively. 
 It is clear that  any proper coloring of $G_k$ contains a rainbow matching constructed by greedily picking an available edge from the stars starting 
 with the smallest star and ending with the star of size $k$.
 The lower bound follows from Lemma~\ref{lem:lowersize}.
\end{proof}

%Recall that $M_k$ denotes the matching on $k$ edges. In order to get a rainbow copy of $M_k$ in every proper coloring of a graph $G$ we clearly need $|V(G)| \geq 2k$. It is known that if $k \geq 3$ then every properly colored $K_{2k}$ indeed contains a rainbow $M_k$, see~\cite{LX1}. Note that $K_{2k}$ has roughly $2k^2$ edges. Next we show that roughly $k^2/2$ edges suffice to enforce a rainbow $M_k$.
% 
%\begin{theorem}\label{thm:sizeMk}
% For every $k \geq 2$,  $AR_s(M_k) =\binom{k+1}{2}$.
%\end{theorem}
%\begin{proof}
% For the upper bound consider the graph $G_k$ that is the vertex  disjoint union of $k$ stars on $1, 2, \ldots, k$ edges respectively. 
% It is clear that  any proper coloring of $G_k$ contains a rainbow matching constructed by greedily picking an available edge from the stars starting 
% with the smallest star and ending with the star of size $k$.
% The lower bound follows from Lemma~\ref{lem:lowersize}.
%\end{proof}

Next let $AR^*_s(H)$ be the smallest number of edges in a \emph{not necessarily simple} graph $G$ with $G \to H$, that is, every proper edge-coloring of $G$ contains a rainbow copy of $H$. Clearly we have $AR^*_s(H) \leq AR_s(H)$ for every graph $H$. Indeed, we can prove slightly better upper bounds when graphs with multiple edges are allowed.

\begin{theorem}\label{thm:multigraphs}
 For any graph $H$ with $k$ edges $AR^*_s(H) \leq \binom{k+1}{2}$,  ~$AR^*_s(P_k) \leq \binom{k}{2}+1$, ~  $AR^*_s(C_k) \leq \binom{k}{2}$,  and $AR^*_s(M_k) = \binom{k+1}{2}$.
\end{theorem}
\begin{proof}
 Let $H$ be any graph with $k$ edges. We label the edges of $H$ by $e_1, \ldots, e_k$ and define $G$ by replacing edge $e_i$ by $i$ parallel edges for all $1\leq i\leq k$. Clearly, $|E(G)| = \binom{k+1}{2}$ and  $G \to H$.

 For the path $P_k$ we take a $P_k=(e_1,\ldots,e_k)$ and define $G$ by replacing edge $e_i$ by $i-1$ parallel edges for $i\geq 2$. The edge $e_1$ remains as it is. Again, it is straight-forward to argue that every proper edge coloring of $G$ contains a rainbow $P_k$. Additionally, every copy of $P_k$ in $G$ has the same start and end vertices.

 For the cycle $C_k$ we first take the graph $G'$ on $\binom{k-1}{2}+1$ edges that forces a rainbow $P_{k-1}$ with end vertices $u$ and $v$. Now a graph $G$ is obtained from $G'$ by introducing $k-2$ parallel edges between $u$ and $v$. Then $|E(G)| = \binom{k-1}{2}+1 + k-2 = \binom{k}{2}$ and $G \to C_k$.

 Finally, observe that the lower bound on $AR_s(H)$ in Lemma~\ref{lem:lowersize} is also a valid lower bound on $AR^*_s(H)$. Together with the first part we conclude that $AR^*_s(M_k) = \binom{k+1}{2}$.
\end{proof}
 
Next, we shall give a general upper bound on $AR_s(H)$ in terms of the size $\tau(H)$ of a smallest vertex cover. We need the following special graph $K(s,t)$ that is the union of a complete bipartite graph with parts $S$ and $T$ of sizes $s$ and $t$, respectively, and a complete graph on vertex set $S$.

\begin{theorem}
 Let $H$ be an $n$-vertex graph with a vertex cover of size $\tau= \tau(H)$. Then there exists an absolute constant $c >0$ such that
 $$AR_s(H) \leq AR_s(K(\tau,n-\tau)) \leq cn\tau^5/ \log\tau + c \tau^6/ \log \tau.$$
\end{theorem}
\begin{proof}
 Let $G=K(c_2 \tau^3/\log \tau,~    n\tau^2+2 \tau^3 )$  be properly colored, where $c_2$ is from~\eqref{eq:local-complete}. 
We shall show that there is a rainbow copy of $K(\tau,n-\tau)$ in $G$.

 First, by~\eqref{eq:local-complete}, there is subset $S'$ of $S$ with $\tau$ vertices inducing a rainbow clique. Let $C$ be the set of colors used on $G[S']$. Consider the set $E'$ of  edges between $S'$ and $T$ having a color from $C$. Then $|E'|\leq  |C|\cdot \tau = \binom{\tau}{2} \cdot \tau$. Let $T'\subseteq T$ be the set of vertices not incident to the edges in $E'$. So, we have that $|T'| \geq |T|- \binom{\tau}{2} \tau \geq n\tau^2 + \tau^3$. Now, the set of colors induced by $S'$ and the set of colors on the edges between $S'$ and $T'$ are disjoint. We shall greedily choose $n-\tau$ vertices from $T'$ so that they send stars of disjoint color sets to $S'$. Pick the first vertex $v_1$ from $T'$ arbitrarily. There are $\tau$ colors on the edges between $v_1$ and $S'$. There are at most $\tau^2$ edges of such colors between $S'$ and $T'$. Delete the endpoints of these edges from $T'$ and choose the next vertex $v_2$ from the remaining vertices in $T'$. The colors of the edges from $v_2$ to $S'$ are 
different from the colors on the edges from $v_1$ to $S'$. Again, delete at most $\tau^2$ vertices from $T'$, if these vertices are the endpoints of the edges having colors appearing on the $v_2-S'$ star. Proceeding in this manner, we can choose a set $T''$ of $n-\tau$ vertices from $T'$ such that $T''\cup S'$ induces a rainbow copy of $K(\tau, n-\tau)$.
\end{proof}

% {\bf Observation} If graph $H$ is dense, having $n$ vertices and $k=dn^2$ edges for $d>0$, then 
% we claim that $AR_s(H) \geq cnk$, for a constant $c$. This is attained by finding a subgraph on $c'n$ vertices such that, in turn, 
% any subgraph of this one, on $c''n$ vertices induced at least $c'''k$ edges. With this, Lemma~\ref{lem:general1} gives us a desired lower bound.

%%%%%%%%%%%%%%%%%%%%%%%%%%%%%%%%%%%%%%%%%%%  
%%%%%%%%%            %%%%%%%%%%%
%%%%%%%%%  F I R S T - F I T  %%%%%%%%%%%
%%%%%%%%%            %%%%%%%%%%%
%%%%%%%%%%%%%%%%%%%%%%%%%%%%%%%%%%%%%%%%%%%
\section{First-Fit Anti-Ramsey numbers}
\label{sec:first-fit}

\begin{definition}\label{def:H-good}
 Let $G$ and $H$ be two graphs. A proper edge-coloring $c:E(G) \to \mathbb{N}$ of $G$ is called \emph{$H$-good} if
 \begin{enumerate}[label = (\arabic*), ref = (\arabic*)]
 \item $G$ contains a rainbow copy of $H$ and\label{item:rainbow}
 \item for every edge $e$ of color $i$ there are edges $e_1,\ldots,e_{i-1}$ incident to $e$ such that $c(e_j) = j$ for $j = 1,\ldots,i-1$.\label{item:support}
 \end{enumerate}
 A graph $G$ is called \emph{$H$-good} if it admits an $H$-good coloring.
\end{definition}

%%%%%%%%%%%%%%%%%%%%%%%%%%%%%%%%%%%%%%%%%%%
%%%%%%%%%%% G E N E R A L %%%%%%%%%%%%%%%%%
%%%%%%%%%%%%%%%%%%%%%%%%%%%%%%%%%%%%%%%%%%%
\begin{lemma}[First-Fit Lemma]\label{lem:first-fit-lemma}
 Let $H$ be a graph. Then $AR_{FF}(H)$ equals the minimum number of edges in an $H$-good graph.
% 
% Moreover, if $G$ is any $H$-good graph and Painter uses First-Fit Builder can force a rainbow copy of $H$ while presenting at most $|E(G)|$ edges.
\end{lemma}
\begin{proof}
% We shall show that if Painter uses First-Fit, then the game ends with an $H$-good graph and that for any  \emph{any} $H$-good graph
% there is a game completed with it.
% 
% \smallskip

 Clearly, if Painter uses First-Fit then at each state of the game the coloring of the current graph is properly colored  and satisfies~\ref{item:support}. 
 Hence if Painter uses First-Fit and the game ends then Builder has presented an $H$-good graph.
 
 \smallskip
 
 On the other hand, let $G$ be  an $H$-good graph and $c$ is an $H$-good coloring of $G$. Let  Builder present the edges in the order of  non-decreasing colors in $c$. 
 We claim that if Painter uses First-Fit she produces the coloring $c$. This  is true for the first edge in the ordering.  Now when an edge $e$ is presented  with $c(e)=i$,   
 we know  that the colors of all previously presented edges coincide with their colors in $c$. 
 By~\ref{item:support}, there are edges  $e_1,\ldots,e_{i-1}$ incident to $e$  that were presented before $e$ and thus  got their corresponding colors $1, \ldots, i-1$.
Thus $e$ must get color $i$ in the First-Fit setting. 
\end{proof}

\subsection{First Fit and weighted matchings}

The First Fit anti-Ramsey number of any graph $H$ is very closely related to the maximum weight of a matching in rainbow colorings of $H$. More precisely,
 we define the {\bf weight of a matching} $M$ with respect to an edge-coloring $c$ as the sum of colors of edges in $M$, that is, as $\sum_{e \in M} c(e)$. 
 A {\bf maximum matching} for $c$ is a matching $M$ with largest weight with respect to $c$, and a {\bf greedy matching} for $c$
 is a matching $M$ that is constructed by greedily taking an edge with the highest color that is not adjacent to any edge already picked. It is well-known that for any $c$ any greedy matching for $c$ has weight at least half the weight of a maximum matching for $c$~\cite{A}. However, we will not use this fact.

Note that if $c$ is a rainbow coloring of $H$, then the greedy matching for $c$ is unique. For a given graph $H$ we define
$$w_\text{max}(H) = \min_{\text{rainbow colorings $c$}} \; \sum_{e \in M} c(e), \quad \text{$M$ maximum matching for $c$}$$
and
$$w_\text{greedy}(H) = \min_{\text{rainbow colorings $c$}} \; \sum_{e \in M} c(e), \quad \text{$M$ greedy matching for $c$}.$$

\begin{lemma}\label{lem:matching_weigth}
 For any graph $H$ on $k$ edges, chromatic index $\chi'$ and vertex cover number $\tau$,  
 $$\frac{k^2}{4\chi'} \leq \frac{w_\text{max}(H)}{2} \leq AR_{FF}(H) \leq 2 w_\text{greedy}(H) \leq (\tau+1) k.$$

\end{lemma}
\begin{proof}
To prove the lower bound on  $AR_{FF}(H)$,  it suffices by  the First-Fit-Lemma to find a lower bound on  $|E(G)|$ for an $H$-good graph $G$.
Fix any $H$-good coloring $c$ of $G$.  There is a  rainbow coloring of $H$ and thus a  matching $M$ in $H$ with  $\sum_{e \in M} c(e) \geq w_\text{max}(H).$  
By~\ref{item:support} in Definition~\ref{def:H-good} each $e \in E(H)$ is incident to at least $c(e)-1$ other edges of $G$. Since every edge in $E(G) \setminus M$ is adjacent  to at most two edges in $M$ we get
 $$AR_{FF}(H) \geq |E(G)| \geq |M| + \frac{1}{2}\sum_{e \in M} (c(e)-1) = \frac{|M|}{2} + \frac{1}{2}\sum_{e \in M} c(e) \geq \frac{w_\text{max}(H)}{2}.$$
 
To prove that $w_\text{max}(H) \geq k^2/(2\chi')$ consider any rainbow coloring $c $ of $H$. The total sum of colors in $H$ is at least $1+ \cdots + k = \binom{k+1}{2}$.
Since  $H$ splits into $\chi'$ matchings $M_1,\ldots,M_{\chi'}$,  at least one of the matchings has weight at least  ${\binom{k+1}{2}}/{\chi'}$.  
Thus,   $w_\text{max}(H)  \geq {k^2}/{2\chi'}$, proving the first inequality of the lemma.

 \smallskip
 
To find  the upper bound $AR_{FF}(H)$, it is sufficient  by the First-Fit-Lemma  to construct an $H$-good graph $G$ with  
small number of edges.  Start with a rainbow coloring $c$ of $H$ and the greedy matching $M$ for $c$.
 Every edge $e$ in $H$ is either in $M$ or incident to an edge $e' \in M$ with $c(e') > c(e)$.
 We construct a graph $G$ with an $H$-good coloring by taking $H$ and $c$ and adding pendant edges to endpoints of edges in $M$ as follows. For every $e \in M$, every endpoint $v$ of $e$ and every color $c < c(e)$ that is not present at $v$ we add a new pendant edge $e_{c,v}$ of color $c$. Since for every $e \in M$ we added at most $2(c(e) - 1)$ new pendant edges we get
 $$AR_{FF}(H) \leq |E(G)| \leq |M| + \sum_{e \in M} 2(c(e)-1) \leq 2 \sum_{e \in M} c(e) - |M| \leq 2w_\text{greedy}(H).$$
 
 %Note that each edge $e'$ in $E(H) \setminus M$ is counted in this summation by the term $c(e)-1$, where $e$ is an edge of $M$ adjacent to $e'$ and having a larger color than the color of $e'$.

 Finally, we find an upper bound on  $w_\text{greedy}(H)$.  Let the degree of an edge  in a graph be the number of edges adjacent to $e$, including $e$, let $V(e)$ be the set of endpoints of $e$.
 Order the edges $e_1, \ldots, e_k$  of $H$ as follows: $e_1=e_{i_1}$ is an edge having highest degree in $H$, 
 $e_{i_1+1}, \ldots, e_{i_2-1}$ are the edges adjacent to $e_{i_1}$,  $e_{i_2}$ is an edge of highest degree in $H- V(e_{i_1})$,  
 $e_{i_2+1}, \ldots, e_{i_3-1}$ are the edges adjacent to $e_2$ in $H-V(e_{i_2})$, $e_{i_3}$ is an edge or the highest degree in $H - (V(e_{i_1}) \cup V(e_{i_2}))$, and so  on.
 Assign the colors  $k, k-1, \ldots, 1 $ to the edges $e_1, e_2, \ldots, e_k$ respectively.
 Then $e_{i_1}, e_{i_2}, \ldots, e_{i_m}$ is a greedy matching, with the weight $w$. Denoting by $d_j$ the degree of $e_{i_j}$ in $H -(V(e_{i_1}) \cup \cdots \cup V(e_{i_j-1}))$,  $j=2, \ldots, m$, 
 and $d_1$ the degree of $e_{i_1}$ in $H$, we have 
  $w = k + (k- d_1) + (k-d_1-d_2) + \cdots + (k -  d_1-d_2 - \ldots - d_{m-1})$.
  Note also that  $d_1\geq d_2\geq \cdots \geq d_m$ and $d_1+ \cdots + d_{m} = k$.
  Thus 
  \begin{align*}
  w &= k(m-1) - \sum_{i=1}^{m-1} (m-i) d_i = mk - k - m\sum_{i=1}^{m-1} d_i  + \sum_{i=1}^{m-1} id_i \\
  &= mk -k - mk + md_m + \sum_{i=1}^{m-1} id_i  =  
  - k + \sum_{i=1}^{m} id_i  \leq  -k + \left(\sum_{i=1}^{m}i\right) \frac{k}{m} \\
  &= -k + (m+1) k/2 = mk / 2 - k/2= (m-1)k/2.
  \end{align*}
  Since $m\leq \tau$ we have that the $w_\text{greedy}(H)\leq w  \leq  (\tau-1)k/2$.
  %~\\~\\
  \end{proof}

For each of the four inequalities in Lemma~\ref{lem:matching_weigth} it is easy to construct a graph for which this inequality is tight  up to a small additive constant. Moreover, as we shall see later the inequality $k^2/(4\chi') \leq AR_{FF}(H)$ is also asymptotically tight for some graphs like paths or matchings. However, we suspect that the upper bound $AR_{FF}(H) \leq (\tau+1)k$ is not tight.
If $\chi'(H) = \Delta(H)$, for example when $H$ is a forest, then by Lemma~\ref{lem:matching_weigth} we have $k^2/(4\Delta) \leq AR_{FF}(H) \leq (\tau+1)k$. Moreover, for every graph $H$ we have $k/\Delta \leq \tau$. Hence, one might think that $AR_{FF}(H)$ can be sandwiched between $k^2/(4\Delta)$ and $ck^2/\Delta$ or between $c(\tau+1)k$ and $(\tau+1)k$ for some absolute constant $c > 0$. However, this is not the case.

\begin{observation}
 For any constant $c > 0$ there exists forests $H_1, H_2$ each  with $k$ edges, maximum degree $\Delta$, and vertex cover number $\tau$, such that
 $$AR_{FF}(H_1) \geq c\frac{k^2}{\Delta} \qquad \text{and} \qquad AR_{FF}(H_2) \leq c(\tau+1)k.$$
\end{observation}
\begin{proof}
 We consider the graph $H(x,y)$ that is the union of a star on $x$ edges and a matching on $y$ edges for sufficiently large $x$ and $y$. Then   $k = |E(H)| = x+y$, $\Delta = \Delta(H) = x$ and $\tau = \tau(H) = y+1$. It is straightforward to argue that $w_\text{max}(H) = w_\text{greedy}(H) = x+y + \sum_{i=1}^y i$. Thus, by Lemma~\ref{lem:matching_weigth} we have $x/2 + y^2/4 \leq AR_{FF}(H) \leq 2x + y^2$.
  Let $H_1= H(x, x)$ with   $x \geq 16c$.  Then  $AR_{FF}(H_1) \geq x^2/4 \geq 4cx \geq ck^2/\Delta$.
 Let $H_2= H(y^2, y)$,  with   $y \geq 3/c$.  Then $AR_{FF}(H_2) \leq 3y^2 \leq cy^3 \leq c(\tau+1)k$.
\end{proof}

%%%%%%%%%%%%%%%%%%%%%%%%%%%%%%%%%%%%%%%%%%%
%%%%%%%%%%% P A T H %%%%%%%%%%%%%%%%%%%%%%%
%%%%%%%%%%%%%%%%%%%%%%%%%%%%%%%%%%%%%%%%%%%
Next we consider the First-Fit anti-Ramsey numbers of some specific graphs $H$, such as the path $P_k$, the matching $M_k$ and the complete graph $K_n$. From Lemma~\ref{lem:matching_weigth} we get $AR_{FF}(P_k) \leq k^2/2 + k$ and $AR_{FF}(M_k) \leq k^2 + k$. However, in both cases we can do better by a factor of $4$ (c.f. Theorem~\ref{thm:path-FF-UB} and Theorem~\ref{thm:matching-FF-UB}), which is best-possible.

\begin{theorem}\label{thm:path-FF-UB}
 For all $k \geq 1$,  $AR_{FF}(P_k) = \frac{k^2}{8}(1 + o(1))$.
\end{theorem}
\begin{proof}
 The lower bound follows directly from Lemma~\ref{lem:matching_weigth}.
For the upper bound we shall construct a $P_k$-good graph on at most $\frac{k^2}{8} + \mathcal{O}(k)$ edges. This graph is defined inductively. In particular, for every $i \geq 1$ we construct a $P_i$-good graph $G_i$ which contains $G_{i-1}$ as a subgraph. The graph $G_1$ consists of a single edge and for $i>1$ the graph $G_i$ is constructed  from $G_{i-1}$ by adding three new vertices and at most $\lfloor \frac{i-2}{4} \rfloor + 2$ new edges. Thus
 $$AR_{FF}(P_k)
  \leq |E(G_k)| 
  \leq \sum_{i=1}^{k}(\lfloor\frac{i-2}{4}\rfloor+2) 
  \leq \frac{1}{4}\binom{k+1}{2} + k
  \leq \frac{k^2}{8} + 2k.$$
 
 We prove the existence of $G_i$ with the following stronger induction hypothesis.
 
 \begin{claim*}
  For every $i\geq 1$ there exists a graph $G_i$ on at most $\frac{i^2}{8} + \mathcal{O}(i)$ edges together with a fixed proper coloring $c_i$ using colors $1,\ldots,i$ and containing a rainbow path $Q_i = (v_1,\ldots,v_i)$ with respect to these colors, such that the \emph{reverse $\hat{c_i}$ of $c_i$}, which is defined as $\hat{c_i}(e) = i+1 - c_i(e)$ for every $e \in E(G_i)$, is $P_i$-good. Moreover, for every edge $e = v_sv_t$ not on $Q_i$ but with both endpoints on $Q_i$ we have $c_i(e) = s+t$.
 \end{claim*}

 We start by defining $G_1$ to be just a single edge $e_1=v_1v_2$ with $c_1(e_1) = 1$. Having defined $G_{i-1}$ we define the graph $G_i$ as being a supergraph of $G_{i-1}$. For every edge $e \in E(G_{i-1}) \subset E(G_i)$ we set $c_i(e) = c_{i-1}(e)$, i.e., edges in the subgraph $G_{i-1}$ have the same color in $c_i$ as in $c_{i-1}$. The vertices in $V(G_i) \setminus V(G_{i-1})$ are denoted by $u_i,u_i'$ and $v_{i+1}$. The edges in $E(G_i) \setminus E(G_{i-1})$ form a matching and are listed below. Each such edge is assigned color $i$ in $c_i$.
 \begin{itemize}
  \item The edge $e_i = v_iv_{i+1}$ extends the path $Q_{i-1}$ to the longer path $Q_i$.
   
  \item The edges $v_{2j}v_{i-2j}$ for $j = 1,\ldots,\lfloor\frac{i-2}{4}\rfloor$. These edges do not coincide with any edge on $Q_i$ since their endpoints have distance at least two on $Q_i$. Moreover, since $c_i(v_{2j}v_{i-2j}) = i = 2j + (i-2j)$ these edges do not coincide with any edge in $E(G_{i-1}) \setminus E(Q_i)$.
  
  \item There are up to two edges $e_i',e_i''$, each incident to a vertex $v_j$ with $j$ even and $2\lfloor\frac{i-2}{4}\rfloor < j < i-2\lfloor\frac{i-2}{4}\rfloor$. The other endpoint of $e_i'$ and $e_i''$ is $u_i$ and $u_i'$, respectively.
 \end{itemize}
%  ~\\
%** MA  fix picture to replace $v, v'$ with $u_i, u_i'$ and $e, e'$ with $e_i', e_i''$  **\\

 \begin{figure}[htb]
  \centering
  \includegraphics{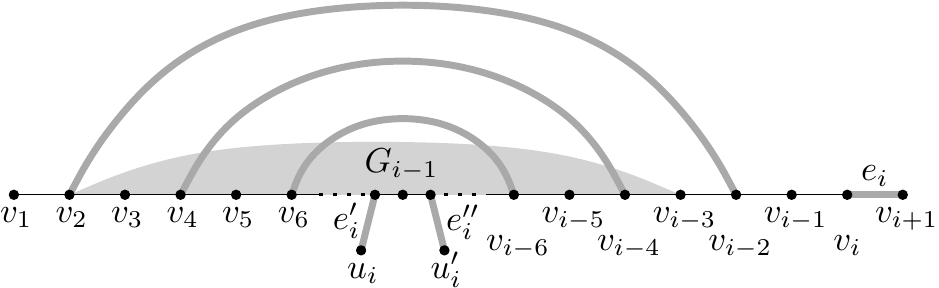}
  \caption{Definition of the graph $G_i$ on basis of $G_{i-1}$: All edges in $E(G_i) \setminus E(G_{i-1})$ are drawn thick and have color $i$ in $c_i$.}
  \label{fig:path-FF-UB}
 \end{figure}
 
 See Figure~\ref{fig:path-FF-UB} for an illustration of how the graph $G_i$ arises from $G_{i-1}$. It is easy to see that for $j=1,\ldots,i$ we have $c_i(e_j) = j$, i.e., that $Q_i$ is a rainbow copy of $P_i$ with respect to  $c_i$.
 
 \smallskip

 The rest of the proof is a  routine check that  the reverse $\hat{c_i}$ of $c_i$ is $P_i$-good, i.e., that for every edge $f\in E(G_i)$ of color $j$ in $c_i$ there exist edges $e_{j+1},\ldots,e_i$ adjacent to $f$ with colors $j+1,\ldots,i$ in $c_i$, respectively. This is clearly the case for $G_1$, which consists of a single edge. Assuming that the reverse of $c_{i-1}$ in $G_{i-1}$ is $P_{i-1}$-good, it then suffices to show that every edge in $E(G_{i-1})$ is incident to at least one edge of color $i$ in $c_i$, that is, at least one edge from $E(G_i) \setminus E(G_{i-1})$.

 First consider an edge $e_j$ from $Q_{i-1}$. If $e_j$ is incident to a vertex $v_s$ with $s$ even and $s \leq 2\lfloor\frac{i-2}{4}\rfloor$ then $e_j$ is incident to the edge $v_sv_{i-s}$ of color $s + (i-s) = i$. If $e_j$ is incident to a vertex $v_s$ with $s$ even and $2\lfloor\frac{i-2}{4}\rfloor < s < i-2\lfloor\frac{i-2}{4}\rfloor$ then $e_j$ is incident to $e$ or $e'$, both of color $i$. If $e_j = e_{i-1}$ then it is incident to $e_i$. If none of the cases applies then $e_j$ is incident to a vertex $v_{i-2s}$ with $s \in \{1,\ldots,\lfloor\frac{i-2}{4}\rfloor\}$ and hence incident to the edge $v_{2s}v_{i-2s}$ of color $2s + (i -2s) = i$.

 Finally, consider any edge $f \in E(G_{i-1}) \setminus E(Q_i)$. Then $f$ is incident to some vertex $v_{2j}$ with $j < c_i(e)-2\lfloor\frac{c_i(e)-2}{4}\rfloor \leq i-2\lfloor\frac{i-2}{4}\rfloor$. Hence $f$ is incident to either $e$, $e'$ or the edge $v_{2j}v_{i-2j}$, all of which have color $i$.
\end{proof}

%%%%%%%%%%%%%%%%%%%%%%%%%%%%%%%%%%%%%
%%%%%%%%%%% C Y C L E %%%%%%%%%%%%%%%
%%%%%%%%%%%%%%%%%%%%%%%%%%%%%%%%%%%%%

\begin{theorem}\label{thm:cycle-FF-UB}
 For all $k\geq 3$, $AR_{FF}(C_k) = \frac{k^2}{8}(1 + o(1))$.
\end{theorem}
\begin{proof}
 The lower bound follows from Theorem~\ref{thm:path-FF-UB} and the fact that $AR_{FF}(C_k) \geq AR_{FF}(P_{k-1})$.
 For the upper bound we first use Theorem~\ref{thm:path-FF-UB} to force a rainbow copy  $P$ of  $P_{k-2}$. Let its end vertices be $u$ and $v$.  
 Now we introduce a set $A$ of $2k$  new vertices  and all edges between $A$ and $\{u,v\}$.  Each vertex in $A$  together with $P$ forms  a copy of $C_k$.
 There are at most $2(k-2)$ edges between $A$ and $\{u,v\}$ of colors from $P$. Thus there is a vertex $w\in A$ such that the colors of $uw$ and $uv$ do not 
 appear in $P$. Thus $P$ and $w$ induce a rainbow $C_k$.
\end{proof}

%%%%%%%%%%%%%%%%%%%%%%%%%%%%%%%%%%%%%%%%%%%
%%%%%%%%%%% M A T C H I N G %%%%%%%%%%%%%%%
%%%%%%%%%%%%%%%%%%%%%%%%%%%%%%%%%%%%%%%%%%%
\begin{theorem}\label{thm:matching-FF-UB}
 For all $k \geq 1$,  $AR_{FF}(M_k)=\lceil\frac{k^2}{4}\rceil$.
 Moreover, there is a bipartite  $M_k$-good graph on $\lceil\frac{k^2}{4}\rceil$ edges.
\end{theorem}
\begin{proof}
 The lower bound follows immediately from Lemma~\ref{lem:matching_weigth}. 
  For the upper bound, which is similar to the one in the proof of Theorem~\ref{thm:path-FF-UB}, we shall prove the following stronger claim by induction on $i$.

% **MA here this should be written in terms of i, the same way as for the paths, the new vertices should have indices **

 \begin{claim*}
  For every $i\geq 1$ there exists a bipartite graph $G_i$ with $\lceil\frac{i^2}{4}\rceil$ edges together with a fixed proper coloring $c_i$ using colors $1,\ldots,i$ and containing a rainbow matching $Q_i = \{e_1,\ldots,e_i\}$, $e_i = u_iv_i$, with respect to these colors, such that the reverse $\hat{c_i}$ of $c_i$ is $M_i$-good. Moreover, for every edge $e=u_sv_t$ not in $Q_i$ but with both endpoints in $Q_i$ we have $c_i(e) = s+t$.
 \end{claim*}
 
 The graph $G_1$ consists of just a single edge $e_1 = u_1v_1$ with $c_1(e_1) = 1$ and $Q_1 = \{e_1\}$. Having defined $G_{i-1}$ we define the graph $G_i$ to be a supergraph of $G_{i-1}$, where we set $c_i(e) = c_{i-1}(e)$ for each edge $e \in E(G_{i-1}) \subset E(G_i)$ and $c_i(e) = i$ for each edge $e \in E(G_i) \setminus E(G_{i-1})$. The vertices in $V(G_i) \setminus V(G_{i-1})$ are denoted by $u_i,v_i$ and $w_i$. The edges in $E(G_i)\setminus E(G_{i-1})$ form a matching and are listed below.
 \begin{itemize}
  \item The edge $e_i = u_iv_i$ extends the matching $Q_{i-1}$ to the larger matching $Q_i$.
  
  \item The edges $u_jv_{i-j}$ for $j=1,\ldots,\lfloor\frac{i-1}{2}\rfloor$. These edges do not coincide with any edge of $Q_i$ since the indices of endpoints differ by at least one. Moreover, since $c_i(u_jv_{i-j}) = j + (i-j) = i$ these edges do not coincide with any edge in $E(G_{i-1})\setminus E(Q_i)$.
 
  \item In case $\lfloor\frac{i-1}{2}\rfloor < i-\lfloor\frac{i-1}{2}\rfloor-1$ there is an edge $e= u_jw_i$ with $j = \lfloor\frac{i-1}{2}\rfloor+1$.
 \end{itemize}
 \begin{figure}[htb]
  \centering
  \includegraphics{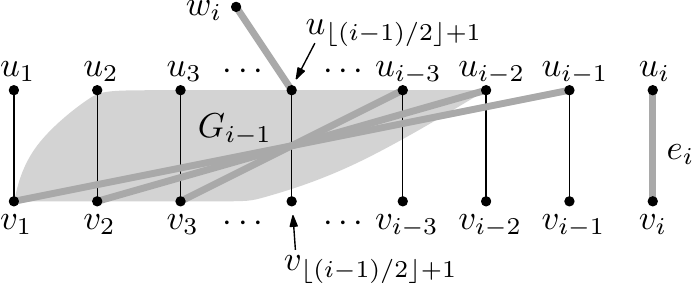}
  \caption{Definition of the graph $G_i$ on basis of $G_{i-1}$: All edges in $E(G_i) \setminus E(G_{i-1})$ are drawn thick and have color $i$ in $c_i$.}
  \label{fig:matching-FF-UB}
 \end{figure}

 See Figure~\ref{fig:matching-FF-UB} for an illustration of how the graph $G_i$ arises from $G_{i-1}$. Similarly to the proof of Theorem~\ref{thm:path-FF-UB} we argue that the reverse of $c_i$ is $M_i$-good by showing that every edge $e \in E(G_{i-1})\subset E(G_i)$ is incident to at least one edge of color $i$ in $c_i$.
 
 First consider an edge $e_j =u_jv_j$ from $Q_{i-1}$. If $j \leq \lfloor\frac{i-1}{2}\rfloor$ then it is incident to the edge $u_jv_{i-j}$ of color $i$. If $j \geq i-\lfloor\frac{i-1}{2}\rfloor$ then $e_j$ is incident to the edge $u_{i-j}v_j$ of color $i$. Finally, if $\lfloor\frac{i-1}{2}\rfloor < j < i-\lfloor\frac{i-1}{2}\rfloor$ then $e_j$ is incident to the edge $v_jw_i$ of color $i$.
 
 Now consider any edge $e \in E(G_{i-1})\setminus E(Q_i)$. Then $e$ is incident to some vertex $u_j$ with $j < c_i(e) - \lfloor\frac{c_i(e)-1}{2}\rfloor \leq i - \lfloor\frac{i-1}{2}\rfloor$. Hence $e$ is incident to the edge $u_jv_{i-j}$ of color $i$.
 
 Finally, note that the number of edges of color $j$ in $c_j$ is exactly $\lceil\frac{j-1}{2}\rceil + 1$. Thus $|E(G_i)| = \sum_{j=1}^i (\lceil\frac{j-1}{2}\rceil + 1) = \lceil \frac{i^2}{4}\rceil$. Moreover, every edge in $G_i$ is incident to one vertex with a $u$-label and one vertex with a $v$-label or $w$-label, that is, $G_i$ is bipartite.
\end{proof}

%%%%%%%%%%%%%%%%%%%%%%%%%%%%%%%%%%%%%%%%%%%
%%%%%%%%%%% C O M P L E T E %%%%%%%%%%%%%%%
%%%%%%%%%%%%%%%%%%%%%%%%%%%%%%%%%%%%%%%%%%%
For $H = K_n$ we get from Lemma~\ref{lem:matching_weigth} that $\frac{n^3}{16} \leq AR_{FF}(K_n) \leq \frac{n^3}{2}$. However, we can do better and indeed determine the value up to lower order terms exactly.

\begin{theorem}
 For all $n \geq 2$,  $AR_{FF}(K_n) = \frac{n^3}{6}(1 + o(1))$.
\end{theorem}
\begin{proof}
 By the First-Fit Lemma  it suffices to prove that every $K_n$-good graph has at least $\frac{n^3}{6}(1 + o(1))$ edges and that there is a $K_n$-good graph on at most $\frac{n^3}{6}$ edges.
 
 \smallskip

 For the lower bound let $G$ be any $K_n$-good graph and $c$ be any $K_n$-good coloring of $G$. Fix $Q$ to be any rainbow copy of $K_n$ in $G$. We denote the edges of $Q$ by $e_1,\ldots,e_k$, where $k = \binom{n}{2}$, so that $c(e_i) < c(e_j)$ whenever $i < j$. Clearly, $c(e_i) \geq i$ for every $i=1,\ldots,k$. Next, for any fixed color $i$ we give a lower bound on the number of edges of color $i$ in $G$.
 
 By~\ref{item:support} each of the $k-i$ edges $e_{i+1},\ldots,e_k$ has an adjacent edge of color $i$. Since only two vertices of $Q$ are joined by an edge of color $i$, every edge of color $i$ different from $e_i$ is incident to at most one vertex of $Q$. Moreover, if two vertices $u,v \in V(Q)$ both have \emph{no} incident edge of color $i$ then $c(uv) < i$. In particular, the subset of $V(Q)$ with no incident edge of color $i$ is a clique in the graph $(V(Q),\{e_1,\ldots,e_{i-1}\})$. Thus there are at most $\sqrt{2i}$ such vertices and hence there are at least $n-\sqrt{2i}-1$ edges of color $i$. Thus
 \begin{equation*}
  |E(G)| \geq \sum_{i=1}^{\binom{n}{2}} n - \sqrt{2i}-1 = \frac{n^3}{6}(1 + o(1)),
 \end{equation*}
 which implies $AR_{FF}(K_n) \geq \frac{n^3}{6}(1 + o(1))$.

 \smallskip
 
 For the upper bound we use the inequality $AR_{FF}(H) \leq 2w_\text{greedy}(H)$ from Lemma~\ref{lem:matching_weigth}. That is, we shall find a rainbow coloring $c$ of $H = K_n$ such that for the greedy matching $M$ for $c$ we have $\sum_{e \in M} c(e) \leq n^3/12$. We may assume without loss of generality that $n$ is even. We label the vertices of $K_n$ by $v_1,\ldots,v_n$ and define the rainbow coloring $c$ arbitrarily so that for $i=1,\ldots,n/2$
 \begin{itemize}
  \item edge $v_{2i-1}v_{2i}$ has color $\binom{2i}{2}$ and
  
  \item the colors $1,\ldots,\binom{2i}{2}$ appear in the complete subgraph of $K_n$ on the first $2i$ vertices $v_1,\ldots,v_{2i}$.
 \end{itemize}

 The greedy matching $M$ for $c$ is given by $M = \{v_{2i-1}v_{2i} \;|\; i = 1,\ldots,n/2\}$ and hence we have
 $$AR_{FF}(K_n) \leq 2 w_\text{greedy}(K_n) \leq 2 \sum_{i=1}^{n/2}\binom{2i}{2} \leq 2 \sum_{i=1}^{n/2} \frac{4i^2}{2} \leq 4 \cdot \frac{1}{3}\left(\frac{n}{2}\right)^3 = \frac{n^3}{6}.$$
\end{proof}

%%%%%%%%%%%%%%%%%%%%%%%%%%%%%%%%%%%%%%%%%  
%%%%%%%%%                                                %%%%%%%%%%%%%%%%%
%%%%%%%%%        O N L I N E                       %%%%%%%%%%%%%%%%%
%%%%%%%%%                                                  %%%%%%%%%%%%%%%%%
%%%%%%%%%%%%%%%%%%%%%%%%%%%%%%%%%%%%%%%%%
\section{Online Anti-Ramsey numbers}\label{sec:online}

In this section we consider the online anti-Ramsey number $AR_o$. In particular, Painter is no longer restricted to use the strategy First-Fit. By Theorem~\ref{thm:inequalities} lower bounds for $AR_{FF}$ from the previous section are also lower bounds for $AR_o$, just like upper bounds for $AR_s$ are upper bounds for $AR_o$. Indeed, these lower bounds are the best we have. However, we can improve on the upper bounds.

%%%%%%%%%%%%%%%%%%%%%%%%%%%%%%%%%%%%%%%%%%%
%%%%%%%%%%% G E N E R A L %%%%%%%%%%%%%%%%%
%%%%%%%%%%%%%%%%%%%%%%%%%%%%%%%%%%%%%%%%%%%

\begin{theorem}\label{thm:online-UB}
 For any $k$-edge graph $H$,  $AR_o(H) \leq k^2$.
  Moreover,  $AR_o(P_k) \leq \binom{k}{2} + 1$,  $AR_o(C_k) \leq \binom{k}{2} +1$, and $AR_o(K_n) \leq n^4/4.$
\end{theorem}
\begin{proof}
 Let $H$ be an graph on $k$ edges. Let the vertices be ordered $v_1,\ldots,v_n$ such that given 
  $v_1, \ldots, v_i$, the next vertex $v_{i+1}$ is chosen from $V \setminus \{v_1,\ldots,v_i\}$ such
   that it has the largest number, $d_{i+1}$, of neighbors in $\{v_1,\ldots,v_i\}$. 
   We define Builder's strategy in rounds. After round $i$ the so far presented graph $G_i$ shall consist of a rainbow 
   induced copy of $H[v_1,\ldots,v_i]$, for $i=1,\ldots,n$, and perhaps additional vertices adjacent to some vertices of that
    copy and inducing an independent set. In round $1$ Builder just presents a single vertex $w_1$, i.e., $G_1 = \{w_1\} \cong H[v_1]$. 
    For $i \geq 2$ round $i$ is defined as follows.
 \begin{enumerate}[label=\arabic*.)]
  \item Let $G$ be the so far presented graph and $G_{i-1}= G[w_1,\ldots,w_{i-1}]$ be the rainbow copy of $H[v_1,\ldots,v_{i-1}]$ in $G$, 
  with $w_i$ corresponding to $v_i$. Let $W_i$ be the set of vertices in $\{w_1, \ldots, w_{i-1}\}$ corresponding to the neighborhood of $v_i$.
 
  \item Builder presents a new vertex $x$ together with the edges $x w$, $ w\in W_i$ presenting $xw_j$ before $xw_{j'}$ for $j<j'$ one by one
   as long as Painter does not use the colors used in $G_{i-1}$. As soon as Painter uses a color already present in $G_{i-1}$, 
   Builder repeats the previous procedure of introducing a new vertex and edges from it to $W_i$. If Painter does not use a 
   color present in $G_{i-1}$ on all the edges from a new vertex to $W_{i}$, call this new vertex $w_i$ and observe that $w_1,\ldots,w_i$
    induce a rainbow copy of $H[v_1,\ldots,v_i]$ with $w_i$ corresponding to $v_i$. This finishes $i$-th round.
 \end{enumerate}
 
 Let $G$ be the final graph obtained after the $n$-th round and containing an induced rainbow copy of $H$ on the 
 vertex set $W=\{w_1, \ldots, w_n\}$. Note that $X = V(G) \setminus W$ induces an independent set and sends edges to $W$. 
 Let the colors used on the edges of this rainbow copy of $H$ be called \emph{old}, and all other colors be called \emph{new}. 
 Recall that $v_i$ has $d_i$ neighbors in $\{v_1,\ldots,v_{i-1}\}$, i.e., $d_i=|W_i|$. Note that, with 
 $d_1=0$, we obtain $d_1+ \cdots + d_n = k$.
 We need one more parameter:   $k_i$ is the number of edges in a subgraph induced by $W$ at a last round 
 involving edges incident to $v_i$.    Formally, $k_i= |E(H[v_1,\ldots,v_j])|$, where   $j$  is the largest index s.t.  $v_iv_{j+1}\in E(H)$.
 
 \begin{claim*}
  The edges between $X$ and $W$ can be decomposed into $|W \setminus w_n| = n-1$ graphs, where each such graph consists of its center $w_i$,
  all edges with old color incident to $w_i$ (call these endpoints old neighbors of $w_i$) and all other edges incident to the old neighbors of $w_i$. 
  Moreover, the size of  such a graph with center $w_i$ is  at most $d_{i+1}(k_i-\operatorname{deg}_H(v_i) + 1)$.
 \end{claim*}
 
 \begin{figure}[htb]
  \centering
  \includegraphics{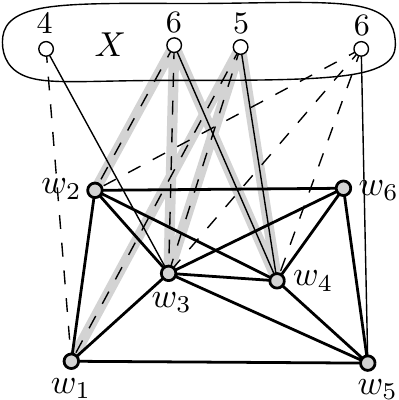}
  \caption{An example graph $G$ after the last round. Edges between $X$ and $W = \{w_1,\ldots,w_6\}$ are drawn dashed if they are new and solid if they are old. The vertices in $X$ are labeled with the number of the round in which they have been introduced. The graph with center $w_4$ is highlighted in gray.}
  \label{fig:general-online-UB}
 \end{figure}
 
 We refer to Figure~\ref{fig:general-online-UB} for an illustration of these graphs. To prove the claim, observe that  each $x\in X$ sends exactly one edge of old color to $W$. Each $w_i\in W$ sends at most $k_i-\operatorname{deg}_H(v_i) + 1$ 
 edges of old colors to $X$. Indeed, when $w_i$ gets its last incident edge in round $j+1$ the total number of old colors is $k_i$, and $\operatorname{deg}_H(v_i) - 1$ 
 of those are used on edges incident to $w_i$ and its neighbors in $\{w_1,\ldots,w_j\}$.
 Moreover, if a vertex $w_i$ is adjacent to a vertex $x\in X$ via an edge of old color, it means that $x$ was introduced at a step $j$, $j\geq i+1$ of the algorithm. Then $x$ is adjacent exactly to the copies of $\{v_1,\ldots,v_i\} \cap N(v_j)$ in $W$. Since by our choice of the ordering of $v_i$s, $|\{v_1,\ldots,v_i\} \cap N(v_j)|\leq |\{v_1,\ldots,v_i\} \cap N(v_{i+1})| = d_{i+1}$, we have that $x$ sends at most $d_{i+1} -1$ edges with new colors to $W$.
 This proves the claim.
 
 \smallskip

 Now, we bound the number of edges in $G$ by the total number of edges in the double stars described above and the number of edges in $H$.
 Here, use the fact $H$ has no isolated vertices,    $k_i \leq k-1$ for all $i=1,\ldots,n-1$, and $\sum_{i=1}^n d_i = k$.
 \begin{equation}
  AR_o(H) \leq k + \sum_{i=1}^{n-1}(k_i-\deg(v_i)+1)d_{i+1} \label{eq:onlineUB}
 \end{equation}
 $$
  \leq k+ \sum_{i=1}^{n-1}(k-1)d_{i+1}
  = k+ (k-1)k
  =k^2.$$
 This  proves the first part of the theorem.
 
 \smallskip
 
 For the special case of $H = P_k$ we take as $v_1$ one of the endpoints of the path, which gives the natural vertex ordering $v_1,\ldots,v_{k+1}$ along the path. Then one easily sees that $d_i = 1$ for each $i =2,\ldots,k+1$ and $k_i = i-1$ for each $i = 1,\ldots,k$. Plugging into~\eqref{eq:onlineUB} we obtain
 \begin{align*}
  AR_o(P_k) &\leq   k + \sum_{i=1}^{k}i - \sum_{i=1}^{k}\deg(v_i) = k + \binom{k+1}{2} - (2k-1) = \binom{k}{2} + 1.
 \end{align*}
  
 Similarly, when $H$ is the cycle $C_k$, we have $d_i = 1$ for $i = 2,\ldots,k-1$, $d_k=2$ and $k_i = i-1$ for each $i=1,\ldots,k-1$, and thus obtain
 \begin{align*}
  AR_o(C_k) & \leq  k + \sum_{i=1}^{k-1}i + (k-1) - \sum_{i=1}^{k-1}\deg(v_i) - 2
   = 2k - 3  + \binom{k}{2} - 2(k-2)
   = \binom{k}{2} + 1.
 \end{align*}
 
Applying the bound $AR_o(H) \leq k^2$ in the first part of the theorem  to $H = K_n$ we immediately  get the desired upper bound $AR_o(K_n)\leq n^4/4$.
\end{proof}

We remark that a more careful analysis for $H = K_n$ using the inequality~\eqref{eq:onlineUB} as we did it for $H = P_k$ and $H = C_k$ in the proof of Theorem~\ref{thm:online-UB} gives asymptotically the same bound as the one above.  When $H$ is the matching $M_k$ then a straightforward  application of Theorem~\ref{thm:online-UB} gives only $AR_o(M_k) \leq k^2$, while from inequality~\eqref{eq:onlineUB} we get $AR_o(M_k) \leq \binom{k+1}{2}$. However, we also get this bound from Theorem~\ref{thm:inequalities} and Theorem~\ref{thm:sizePathCycle} as follows: $AR_o(M_k) \leq AR_s(M_k) = \binom{k+1}{2}$.

\smallskip

Let us define $AR_o^*(H)$ just like $AR_o(H)$ but with the difference that Builder is allowed to present parallel edges. Note that from Theorem~\ref{thm:sizePathCycle} and Theorem~\ref{thm:multigraphs} follows that $AR_s^*(M_k) = AR_s(M_k)$, i.e., allowing multiple edges does not help here. However, in the online setting we get a better upper bound.  

\begin{theorem}
 For all $k \geq 1$, $AR_o^*(M_k)\leq \frac{k^2}{3} + \mathcal{O}(k)$.
\end{theorem}
\begin{proof}
   We present edges in $k$ phases, introducing at most $2\lceil\frac{i}{3}\rceil$ edges in phase $i$, for each $i=1,\ldots,k$ such that 
   for any proper coloring of Painter and  every $i=1,\ldots,k$, the following holds.
 At each stage we identify a pair $(F_i, v_i)$, where 
 \begin{itemize}
  \item $F_i$ is a rainbow matching on $i$ edges
  \item $v_i$ is a vertex of degree $\lceil\frac{i}{3}\rceil$ 
  \item $v_i$   not incident to any edge in $F_i$
  \item each edge  incident to  $v_i$ whose color is not already in $F_i$ is  not adjacent to any edge in $F_i$.
 \end{itemize}
  
 In phase $1$ present two independent edges $e,e'$, let  $F_1 = \{e\}$ and  $v_1= e'$. 
 For $i\geq 2$ we call colors of edges in $F_i$ old colors and all other colors new colors. Introduce two bundles of $\lceil\frac{i}{3}\rceil$ parallel edges each, one after another. 
 As endpoints of the first bundle we choose $v_{i-1}$ and a new vertex $v$.   First assume that Painter colors in such a way that some edge $e$ at $v_{i-1}$ has a new color. 
 Then we choose the endpoints of the second bundle to be two new vertices $u,w$ and let $F_i = F_{i-1} \cup e$ and $v_i = u$.
 Otherwise, each edge at $v_{i-1}$ has an old color. Then we choose the endpoints of the second bundle to be $v_{i-1}$ and a  new vertex $z$. Since $v_{i-1}$ now has degree $\lceil\frac{i-1}{3}\rceil + 2\lceil\frac{i}{3}\rceil \geq i$ at least one edge $e$ of the second bundle receives a new color.  Set $F_i = F_{i-1} \cup e$ and $v_i =v$.
  
 Since in each phase we introduce $2\lceil\frac{i}{3}\rceil$ edges we get   $AR_o^*(M_k) \leq \sum_{i=1}^k 2\left\lceil\frac{i}{3}\right\rceil  \leq   \frac{k^2}{3} + \mathcal{O}(k).$
\end{proof}

\section{Concluding Remarks and Open Questions.}

This paper investigates three anti-Ramsey type functions:  $AR_{FF}(H)$, $AR_o(H)$ and $AR_s(H)$.
Although these functions might be identical for some graphs, such as stars, in general these are distinct 
functions.  For example,   $AR_{FF}$ and $AR_o$ differ from  $AR_s$ already in their order of magnitude in case of $K_n$. 
For  paths or matchings, all three numbers  have the same order of magnitude.

%We proved that each of $AR_{FF}(H)$, $AR_o(H)$ and $AR_s(H)$ is bounded from above and below in terms of $AR_{loc}(H)$, provided $H$ has no isolated vertices. 
% The inequality $AR_s(H) \leq \binom{AR_{loc}(H)}{2}$  is strict only if the smallest in size graph $G$, for which $G\to H$ is  not complete.
 In classical size Ramsey theory, in particular in an argument attributed to Chv\'atal, see ~\cite{EFRS}, it is claimed that if $G$ is a graph with the smallest number of edges such that any $2$-coloring contains a monochromatic copy of $K_n$ then $G$ must be a complete graph.  So one might think that in order to force a rainbow clique, i.e., taking $H = K_n$, it is most efficient to color a large clique, i.e., $G = K_N$, instead of some other graph $G$.  That would imply that $AR_s(K_n) \leq \binom{AR_{loc}(K_n)}{2}$.   We prove that this intuition and analogue of size-Ramsey theorem does not hold already for $K_4$.
 The following general question remains:

\begin{prob}
 By how much do  $AR_s(K_n)$ and $\binom{AR_{loc}(K_n)}{2}$  differ as $n$ goes to infinity?
\end{prob} 

We believe that the inequality $AR_s(H) \leq \binom{AR_{loc}(H)}{2}$ could  be asymptotically tight for $H = K_n$ and for some sparse graphs $H$.  When $H$ is the path $P_k$, then the best-known upper bound for $AR_s(P_k)$ is given by  the best-known upper bound for $AR_{loc}(P_k)$.   Next is: 

\begin{prob}
 Improve the upper bound for $AR_s(P_k)$.
\end{prob}

We proved that  $AR_{FF}(H)$ is very closely related to a min-max quantity with respect to matchings in $H$. We  derived an upper bound  $AR_{FF}(H) \leq (\tau+1)k$, which we suspect could be improved by the factor of $2$. 

\begin{prob}
 Is it true that for every graph $H$,  $AR_{FF}(H) \leq (\tau+1)k/2$?
\end{prob}

For  all our lower bounds for $AR_o$ Painter uses the First-Fit strategy.
This gives rise to the following:

\begin{prob}
 Find a class of graphs for which $AR_o$ and $AR_{FF}$ have different asymptotic behavior.
\end{prob}

The theme of this paper was to efficiently force a rainbow copy of a specific graph $H$ in every proper coloring of a constructed graph $G$. We have measured efficiency by the number of edges required in $G$ and have considered this both in the online and offline setting. Let us mention another concept of efficiency that might be interesting to study: For given graph $H$ what is the \emph{least maximum degree} of a graph $G$ with $G \to H$, i.e., where every proper edge-coloring of $G$ contains a rainbow copy of $H$? Clearly, the maximum degree of $G$ must be at least $|E(H)|-1$ since otherwise $G$ can be properly colored with less than $|E(H)|$ colors and hence does not contain a rainbow copy of $H$.

There is also an online variant of this question, which is analogous to the online anti-Ramsey numbers we defined here. One can show that Builder can force a rainbow matching on $k$ edges even if the graph she presents has maximum degree at most $\lceil (k+1)/2 \rceil$ and that no rainbow copy of any $k$-edge graph $H$ can be forced by Builder if she is restricted to presenting a graph of maximum degree strictly less than $\lceil (k+1)/2 \rceil$.

\begin{prob}
 How large rainbow paths can Builder force, when she presents a graph with bounded maximum degree?
\end{prob}

We remark that such questions are closely related to the so-called \emph{restricted setting} considered by Pra{\l}at~\cite{P} and Grytczuk, Ha{\l}uszczak and Kierstead~\cite{GHK}.

\bibliography{antir}
\bibliographystyle{plain}

\end{document}